\documentclass[final,3p,times]{elsarticle}
\usepackage{amssymb}
\usepackage[utf8]{inputenc}
\usepackage{textcomp}
\usepackage{ulem}
\usepackage{graphicx}
\usepackage{amsmath}
\usepackage{bm}
\usepackage[version=4]{mhchem}
\usepackage{siunitx}
\usepackage{color}
\usepackage{mathtools}
\usepackage{float}
\usepackage{caption,subcaption}
\usepackage{tikz}
\usepackage[many]{tcolorbox}
\usepackage{color}
\usepackage{epstopdf}
\usepackage{algorithm}
\usepackage{algpseudocode}
\usepackage{setspace}
\usepackage{bm}
\newcommand{\defeq}{\vcentcolon=}

\usepackage{amsthm} %  proof 
\newtheorem{lemma}{Lemma}

\newcommand{\bo}{\boldsymbol}
\DeclareMathOperator{\atantwo}{atan2}
\newcommand{\overbar}[1]{\mkern1.5mu\overline{\mkern-1.5mu#1\mkern-1.5mu}\mkern1.5mu}
\makeatletter
\usetikzlibrary{calc}

\usepackage{epsfig} % for postscript graphics files
\usepackage{times} % assumes new font selection scheme installed
\usepackage{hyperref}
\usepackage{cleveref}
\usepackage{mathtools}
\usepackage{siunitx}

\journal{arxiv}

\begin{document}

\begin{frontmatter}


\title{Nonlinear Optimal Guidance for Impact Time Control with Field-of-View Constraint}
% \author{}
\author{Fangmin Lu\footnote{PhD student, School of Aeronautics and Astronautics.},
  Zheng Chen\footnote{Researcher, School of Aeronautics and Astronautics;
    email: \underline{z-chen@zju.edu.cn}. Member AIAA.},
    Kun Wang\footnote{PhD, School of Aeronautics and Astronautics.}.
}

\affiliation{organization={Zhejiang University},
            addressline={Zheda Road 38},
            city={Hangzhou},
            postcode={310027},
            state={Zhejiang},
            country={China}}

% \title{}


% \affiliation{organization={},%Department and Organization
% \title{Nonlinear Optimal Guidance for Impact Time Control with Field-of-View Constraint}
% \author{}
\begin{abstract}
An optimal guidance law for impact time control with field-of-view constraint is presented. The guidance law is derived by first converting the inequality-constrained nonlinear optimal control problem into an equality-constrained one through a saturation function. Based on Pontryagin's maximum principle, a parameterized system satisfying the necessary optimality conditions is established. By propagating this system, a large number of extremal trajectories can be efficiently generated. These trajectories are then used to train a neural network that maps the current state and time-to-go to the optimal guidance command. The trained neural network can generate optimal commands within 0.1 milliseconds while satisfying the field-of-view constraint. Numerical simulations demonstrate that the proposed guidance law outperforms existing methods and achieves nearly optimal performance in terms of control effort.
\end{abstract}
% \begin{highlights}
% \item A neural network-based impact time control guidance with field-of-view constraint.
% \item A parameterized system for generating extremal trajectories is established.
% \item Smoothness of the guidance command is guaranteed.
% \item Time-to-go estimation is not required.
% \end{highlights}


\end{frontmatter}


\section{Introduction}\label{section:1}
Since the 1950s, Proportional Navigation (PN) has been widely applied to designing guidance laws owing to its simplicity, effectiveness, and ease of implementation~\cite{zarchan2012tactical}. However, it is well known that PN cannot control the duration of pursuing process, which makes it difficult for multiple pursers to cooperatively achieve a target. For this reason, significant effort has been devoted to designing Impact Time Control Guidance (ITCG), and papers related to ITCG have been increasingly published in recent decades; see, e.g.,~\cite{Jeon:2006,jeon2016impact,2016Modified, chen2019modified, DongVarying, kumar2013sliding,Cho2016Nonsingular,2018Nonsingular,hu2019sliding, Lyapunov2016, tekin2017polynomial, tekin2017adaptive, tsalik2019circular, zhang2014impact, kim2016impact,singh2024time, Dong2022article}. 

Jeon I.-S. {\it et al.} in~\cite{Jeon:2006} were probably among the first to study the ITCG; the impact time was adjusted by adding a biased term relative to the impact time error to the PN. Hence, the estimation of time-to-go was required, and the authors in~\cite{Jeon:2006} derived the time-to-go under the assumption that the kinematics is linear. In the subsequent paper~\cite{jeon2016impact}, a natural extension to the estimation of time-to-go was presented by using nonlinear kinematics. In order to precisely control the impact time in scenarios with large initial heading errors, the pure PN was modified in~\cite{2016Modified} by Cho {\it et al.} by using a time-varying gain, for which a closed form of time-to-go was derived in the nonlinear setting.
Recently, Dong {\it et al.}~\cite{DongVarying} proposed a varying-gain PN to control the impact time and extended its applications to three-dimensional engagement scenarios.

In addition to the PN-based ITCG, some other methods have been developed in the context of nonlinear control theory, primarily categorized as Sliding-Mode-Control (SMC)-based~\cite{kumar2013sliding,Cho2016Nonsingular,2018Nonsingular,hu2019sliding} and Lyapunov-based~\cite{Lyapunov2016} guidance laws. Kumar {\it et al.}~\cite{kumar2013sliding} chose the sliding surface corresponding to the Line-of-Sight (LOS) rate and the impact time error, and then derived the guidance law based on conventional and improved estimation of time-to-go. Through introducing the lead angle into a Lyapunov candidate function,  Saleem {\it et al.}~\cite{Lyapunov2016} derived the impact time in terms of a Beta function of the initial heading error and initial range. The impact time can then be controlled by varying a single parameter.
A common limitation of guidance laws such as those in~\cite{kumar2013sliding,Cho2016Nonsingular,Lyapunov2016} is the need for a large initial guidance command, primarily due to significant errors in time-to-go estimation or lead angle. This issue results in the requirement of more control effort and degradation of the performance.

Alternatively, trajectory shaping provides another approach for designing ITCG laws. In \cite{tekin2017polynomial}, the quadratic and cubic polynomials were employed to follow the look-angle profile without relying on time-to-go estimation. The coefficients of the polynomials were determined through an integration equation. Subsequently, the methodology introduced in \cite{tekin2017polynomial} was extended to scenarios with varying speed, as detailed in \cite{tekin2017adaptive}. The main idea of the latter work was to employ an $n$th-order polynomial to approximate the lead angle. Under the assumption of the linearized setting, an analytical solution of the guidance gain was derived as a function of the range, the lead angle, and time-to-go. Whereas in \cite{tsalik2019circular}, the impact time control was achieved by imposing the curvature of a circular trajectory. The downside of trajectory shaping methods presented in~\cite{tekin2017polynomial,tekin2017adaptive,tsalik2019circular} 
is the lack of optimality.
% is that they commonly need to solve 
% a set of nonlinear equations, which may bring convergence issues.

In practical applications, the pursuers are subject to the Field-of-View (FOV) 
constraint. Violating the FOV constraint can lead to a loss of target lock, which may  further result in interception failure. Consequently, recent researches on ITCG have been focused on incorporating the FOV constraint~\cite{zhang2014impact,kim2016impact,2018Nonsingular,singh2024time,tekin2017polynomial,Dong2022article,kang2024impact}. In \cite{zhang2014impact}, a rule involving the cosine of the weighted lead angle was used in the biased term, guaranteeing satisfaction of the FOV constraint. While in \cite{kim2016impact}, the authors formulated the biased term by applying the backstepping control theory. Specifically, they controlled the relative velocity of the target and the pursuer along the LOS to satisfy the FOV constraint. In addition to the PN-based guidance laws, Chen {\it et al.}~\cite{2018Nonsingular} presented an SMC-based ITCG law under the FOV limit. In that work, a virtual variable of the desired lead angle was introduced by solving an algebraic-trigonometric equation. Then, the error between the actual lead angle and the desired one was defined as the sliding surface. The guidance law was further extended to incorporating the FOV constraint by redefining the sliding surface with a switch logic, rendering discontinuity at the exit of the constrained arc. Singh {\it et al.}~\cite{singh2024time} employed a barrier Lyapunov function to derive the guidance law satisfying the FOV bound with a first-order autopilot dynamics. In \cite{tekin2017polynomial}, the range was formulated as a polynomial with respect to time, and the guidance law had a switch logic to meet the FOV constraint. However, this method may generate a discontinuous guidance command at the switching instant. It is worth mentioning that a united method for FOV-limited guidance was proposed in \cite{Dong2022article}. The authors showed that the FOV limit can be converted to a biased command, and then they employed this method to 2D and 3D scenarios with impact time and angle control.

Although the guidance laws mentioned in the preceding paragraph could effectively handle the FOV constraint, ensuring the optimality of the control effort still remains a challenge. Indeed, PN with a navigation gain of three proves to be optimal under the linearized setting and certain ancillary assumptions, such as linear approximation of the LOS angle~\cite{Bryson:69}. However, PN could suffer severe optimality degradation given a large initial heading error. It has been demonstrated that the incorporation of the nonlinear setting into PN-based guidance laws results in enhanced guidance performance~\cite{2016Modified,chen2019nonlinear,Jeon:10}. This motivates researchers to introduce nonlinear optimal control problems (OCPs) in guidance law design. However, conventional methods for solving nonlinear OCPs, such as direct and indirect methods, suffer from convergence issues and computational burdens~\cite{Betts1998}. Consequently, their applications in real-time pursuer guidance are deemed impractical.

To address these challenges, recent advancements have explored the use of Neural Networks (NNs) to formulate feedback optimal guidance within the aerospace engineering domain~\cite{you2021learning,wang2024fuel,Izzo2021,wang2022nonlinear, wang2024real}. Notably, Wang {\it et al.} proposed a nonlinear optimal guidance law with impact time constraint in \cite{wang2022nonlinear}. In that work, the authors established a parameterized system satisfying the necessary conditions according to Pontryagin's Maximum Principle (PMP), and then an NN was trained by the dataset generated from the parameterized system to map the state and time-to-go to the optimal guidance command.

In this paper a novel method for generating feedback solutions of a nonlinear OCP with impact-time and FOV constraints is proposed. First, the interception problem is built as an inequality-constrained nonlinear optimal control problem (ICNOCP), and saturation functions proposed by Graichen {\it et al.}~\cite{graichen2010handling} is employed to convert the ICNOCP into an equality-constrained nonlinear optimal control problem (ECNOCP). Next, the PMP is employed to derive the necessary optimality conditions for the ECNOCP. Then, the parameterization method in \cite{wang2022nonlinear} is used to parameterize the extremal trajectories. By doing so, a large number of extremal trajectories, containing the current state of the pursuer and the corresponding optimal guidance command, can be readily obtained. Furthermore, a neural network trained by the sampled data, is
sufficient to generate the optimal guidance command within a constant time.

The remainder of this paper is organized as follows. In Section~\ref{section:2}, the interception geometry and the optimal control problem are formulated. Section~\ref{section:3} derived the necessary conditions and establish a parameterized system. The dataset generation and the neural network training are presented in Section~\ref{section:4}. In Section~\ref{section:5}, the numerical results are provided to evaluate the performance of the proposed guidance law. Finally, this paper concludes in Section~\ref{section:6}.

\section{Problem Formulation}\label{section:2}
\subsection{Nonlinear Kinematics}
Consider a two-dimensional interception geometry with a stationary target, as shown
in Fig.~\ref{fig:frame}. The origin of the Cartesian inertial frame $Oxy$
is located at the target. The x-axis points to the East, and the y-axis points to the North. Denote by $(x,y)$ the position of the pursuer in the frame $Oxy$, and denote by $\bo{V}$ the velocity vector.  
\begin{figure}[thpb]
  \centering
  % \framebox{\parbox{3in}{ }}
  \includegraphics[scale=1]{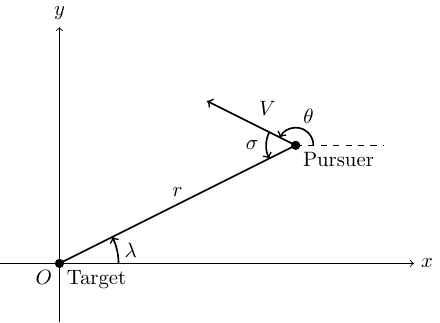}
  \caption{Geometry of a two-dimensional interception.}\label{fig:frame}
\end{figure}
Through normalizing the speed to 1, 
the nonlinear kinematics of the pursuer can be expressed by~\cite{Lu:06}
\begin{align}
&\left\{
  \begin{aligned}
    \dot{x}(t)    &= \cos\theta(t), \\
    \dot{y}(t)     &= \sin\theta(t), \\
    \dot{\theta}(t) & = u(t),
  \end{aligned}
\right.
  \label{eq:kinematics}
\end{align}
where $t$ denotes the time, the over dot denotes the differentiation with respect to time, $\theta$ the heading angle measured counter-clockwise from the x-axis to the velocity vector $\bo{V}$, and $u$ the control variable representing the normal acceleration.
Let $\sigma$ be the lead angle measured clock-wise from the Line of Sight (LOS) to the velocity vector $\bo{V}$, as shown in Fig.~\ref{fig:frame}. It is obvious that the lead angle $\sigma$ can be expressed as
\begin{align}
  \sigma & = \pi + \lambda - \theta,
\end{align}
where $\lambda$ is the LOS angle measured counter-clockwise from the x-axis to the LOS.
During the course for the pursuer to arrive the target, it is required that the target remains in the Field of View (FOV). Thus, the lead angle should be bounded as below
\begin{equation}
  \sigma  \in [-\sigma_M, \sigma_M],
  \label{eq:fov_limit}
\end{equation}
where $\sigma_M\in\left(0,\pi/2 \right]$ denotes the FOV limit.
As the origin of the frame $Oxy$ is located at the target, the final boundary conditions are given as
\[x(t_f) = 0, \ y(t_f) = 0,\]
where $t_f$ is the final time. 

\subsection{Optimal Control Problem}
The optimal control problem consists of steering the kinematics in Eq.~(\ref{eq:kinematics}) from an initial state $(x_0,y_0,\theta_0)$ at $t=0$ to impact the target $(0,0)$ at a fixed terminal time $t_f$ so that the FOV constraint in Eq.~(\ref{eq:fov_limit}) is met and the control effort is minimized. Let $\bo{z}=(x,y,\theta)$ denote the state vector, then the OCP can be formulated as
\begin{align}\label{eq:OCP}
\mathcal{P}_0:
  \begin{cases}
     & \underset{u}{\min}~\displaystyle \int_{0}^{t_f} \frac{1}{2} u^2 \mathrm{d}t \\
     & \rm{s.t.}                                       \\
     & \dot{\bo{z}} = \bo{f}(\bo{z}),                  \\
     & -\sigma_M \leq \sigma \leq \sigma_M,            \\
     & x(0) = x_0,\ y(0) = y_0,\ \theta(0) = \theta_0,
     \\
     & x(t_f) = 0,\ y(t_f) = 0,
  \end{cases}
\end{align}
where $\bo{f}(\cdot)$ indicates the right hand side of the nonlinear kinematics equations in Eq.~(\ref{eq:kinematics}). 

Due to the FOV constraint, the extremal trajectory may consist of constraint-active and -inactive arcs, making it challenging to know {\it a priori} the structure of extremal trajectories~\cite{malisani2016interior}. In the following paragraphs, we shall introduce a saturation function, so that $\mathcal{P}_0$ in Eq.~(\ref{eq:OCP}) is transformed to an equality-constrained OCP. 

According to the engagement geometry in Fig.~\ref{fig:frame}, we have 
\begin{align}
  \begin{split}
    \cos \sigma &=  -\frac{x(t)\cos\theta(t)+y(t)\sin\theta(t)}{\sqrt{{x(t)}^2+{y(t)}^2}}.
  \end{split}
\end{align}
Because of $\cos\sigma \leq 1$, we can rewrite the constraint in Eq.~(\ref{eq:fov_limit}) as 
\begin{align}
  S(\bo{z}) \defeq \cos\sigma_M -  \cos\sigma \leq 0.
  \label{eq:fov_limit_s}
\end{align}
Differentiating the constraint with respect to time leads to
\begin{align}
  \dot{S}(\bo{z}) = \Big(\frac{\sin\sigma}{r} - u \Big)\sin\sigma,
  \label{eq:Sdot}
\end{align}
Notice from Eq.~(\ref{eq:Sdot}) that the control variable appears in the first time deriative of $S(\bo{z})$. Thus, the FOV constraint is of order one. Then, according to~\cite{graichen2010handling}, the inequality constraint can be replaced by 
\begin{align}
    S(\bo{z}) - \psi(\xi) = 0,
    \label{eq:satfun}
\end{align}
where $\psi(\xi) = -\exp(-\xi)$ with $\xi$ being a newly introduced state variable.

Let us denote by $\omega$ the time derivative of $\xi$, i.e., 
\begin{align}
    \dot{\xi} = \omega.
    \label{eq:xidot}
\end{align}
Then, $\omega$ is a new control variable.
Combining Eq.~(\ref{eq:satfun}) with Eq.~(\ref{eq:xidot}), we immediately have 
\begin{align}
    \dot{S}(\bo{z}) = \psi^{\prime}(\xi)\omega,
\end{align}
where the prime denotes the differentiation with respect to $\xi$.
Since the original cost in $\mathcal{P}_0$ is independent of $\omega$,
an additional regularization term should be added. The new regularized cost is as follows:
\begin{align}
  J(u,\omega, \varepsilon) = \frac{1}{2} \int_{0}^{t_f} (u^2 + \varepsilon \omega^2) \mathrm{d}t,
  \label{eq:cost_reg}
\end{align}
where $\varepsilon < 1$ is an adjustable parameter.

The original inequality-constrained $\mathcal{P}_0$ in Eq.~(\ref{eq:OCP}) is reformulated as
an equality-constrained $\mathcal{P}_{\varepsilon}$ as follows:
\begin{align}\label{eq:OCPe}
\mathcal{P}_{\varepsilon}:
  \begin{cases}
     & \underset{u,\omega}{\min} ~\displaystyle\int_{0}^{t_f} \frac{1}{2}(u^2 + \varepsilon \omega^2) \mathrm{d}t \\
     & \rm{s.t.}                                                                                     \\
     & \dot{\bo{z}} = \bo{f}(\bo{z}),                                                                     \\
     & \dot{\xi} = \omega,                                                                           \\
     & x(0) = x_0,\ y(0) = y_0,\ \theta(0) = \theta_0,
     \\
     & x(t_f) = 0,\ y(t_f) = 0,
     \\
     & S(\bo{z}_0) = \psi(\xi_0),                                                                    \\
     & S(\bo{z}_f) = \psi(\xi_f),\\
    &\dot{S}(\bo{z}) -\psi'\omega= 0.
  \end{cases}
\end{align}

The solution of $\mathcal{P}_{\varepsilon}$ converges to that of $\mathcal{P}_0$ as $\varepsilon$ approaches to zero. Readers interested in the corresponding proof are referred to~\cite{graichen2010handling}. As a result, it amounts to addressing $\mathcal{P}_{\varepsilon}$, and the solution of $\mathcal{P}_{\varepsilon}$ will be characterized in the following section. 

\section{Parameterization of the Optimal Solution}\label{section:3}

\subsection{Necessary Conditions}\label{subsec:necessary_conds}
Let $p_x$, $p_y$, $p_\theta$, and $p_\xi$ be the costate variables of $x$, $y$, $\theta$, and $\xi$, respectively. Then, the Hamiltonian is expressed as
\begin{align}
  % \begin{split}
    \mathcal{H} &= p_x \cos\theta + p_y\sin\theta + u p_\theta + p_\xi \omega 
    - \frac{1}{2}u^2 - \frac{1}{2}\varepsilon \omega^2
    + \mu\left[\dot{S}(\bo{z}) - \psi'(\xi)\omega\right],
  % \end{split}
\end{align}
where $\mu$ denotes the Lagrange multiplier.
The costate variables are governed by
\begin{align}
&\left\{
\begin{aligned}
\label{eq:euler}
\dot{p}_x = - \frac{\partial \mathcal{H}}{\partial x}  &= -\mu \frac{\sin\sigma(r\sin\theta - 2x\sin\sigma)}{r^3}, \\
\dot{p}_y  = - \frac{\partial \mathcal{H}}{\partial y}&= -\mu \frac{\sin\sigma(-r\cos\theta - 2y\sin\sigma)}{r^3}, \\
\dot{p}_\theta  = - \frac{\partial \mathcal{H}}{\partial \theta}&= p_x\sin\theta - p_y\cos\theta - \mu \frac{\sin\sigma(x\cos\theta + y\sin\theta)}{r^2},\\
\dot{p}_\xi = - \frac{\partial \mathcal{H}}{\partial \xi} & = \mu\omega \exp \xi,
\end{aligned}
\right.
\end{align}
where $\sin\sigma = (x\sin\theta-y\cos\theta)/r$.
According to the PMP~\cite{Pontryagin}, we have
\begin{align}
&\left\{
\begin{aligned}
\label{eq:age}
  0 = \frac{\partial \mathcal{H}}{\partial u} & = p_\theta - u  + \mu\sin\sigma \\ 
  0 = \frac{\partial \mathcal{H}}{\partial \omega} &= p_\xi - \varepsilon \omega - \mu \exp \xi ,\\
  \begin{split}
    0 = \frac{\partial \mathcal{H}}{\partial \mu} &= \sin\sigma\left(\frac{\sin\sigma}{r}-u\right) - \omega \exp\xi.
    % \label{eq:age3}
  \end{split}
  \end{aligned}
  \right.
\end{align}

Hereafter, the triple $(x(\cdot),y(\cdot),\theta(\cdot))$ on $[0,t_f]$ will be said an extremal trajectory if it is a solution of Eqs.~(\ref{eq:kinematics},\ref{eq:xidot},\ref{eq:euler}) with the necessary conditions in Eq.~(\ref{eq:age}) satisfied. Accordingly, the control along an extremal trajectory will be said an extremal control.
% Since the heading angle $\theta$ is free at the final time, $\xi(t_f) = S^{-1}(\psi(\bo{z}(t_f))$ is also free.
% the transversality condition implies
Since the heading angle $\theta$ and the new introduced state $\xi$ are free at the final time, it follows from the transverality condition that
\begin{align}
&\left\{
  \begin{aligned}
    \label{eq:trans}
  &p_\theta(t_f)  = 0,\\
  &p_\xi(t_f)  = 0.
  \end{aligned}
\right.
\end{align}

\subsection{Additional Necessary Conditions}
Besides the necessary conditions derived in the previous subsection, we can establish an additional optimality conditions through Lemma~\ref{le:align}.
% \textcolor{red}{
\begin{lemma}\label{le:align}
Given any extremal trajectory $(x(\cdot),y(\cdot),\theta(\cdot))$ on $[0,t_f]$,
  if there exists a time ${\hat{t}}\in (0,t_f)$ so that the velocity vector
  is collinear with the LOS, i.e.,
  \begin{align}
|\cos\sigma(\hat{t})| = \frac{|x(\hat{t})\cos(\theta(\hat{t})) + y(\hat{t})\sin(\theta(\hat{t}))|}{\sqrt{{x(\hat{t})}^2 + {y(\hat{t})}^2}} = 1,
    \label{EQ:cos_sigma=1}
  \end{align}
  then the extremal trajectory  $(x(\cdot),y(\cdot),\theta(\cdot))$ on $[0,t_f]$ is not an optimal trajectory.
\end{lemma}
\noindent The proof for Lemma~\ref{le:align} is postponed to \ref{appendix:A}.

\subsection{Parameterization of Extremal Trajectories}
Set $\bo{g} = {[ \frac{\partial \mathcal{H}}{\partial u},~\frac{\partial \mathcal{H}}{\partial \omega},~\frac{\partial \mathcal{H}}{\partial \mu}]}^T$, $\bar{\bo{u}} = {[u, \omega, \mu]}^T$, $\overbar{\bo{z}} = {[x,y,\theta, \xi]}^T$, $\bar{\bo{p}} = {[p_x,~p_y,~p_\theta, p_\xi]}^T$, and $\bar{\bo{f}} = {[\dot{f}^T, \omega]}^T$. We have the following lemma.
\begin{lemma}\label{lemma:jacobian}
    The Jacobian $\dfrac{\partial \bo{g}}{\partial {\bar{\bo{u}}}}$ is non-singular.
\end{lemma}
\noindent The proof of Lemma~\ref{lemma:jacobian} is postponed to \ref{appendix:A}.
As a result of Lemma~\ref{lemma:jacobian}, $\dfrac{\mathrm{d}\bar{\bo{u}}}{\rm{d}t}$ can be expressed as
\begin{align}
  \frac{\mathrm{d}\bar{\bo{u}}}{\mathrm{d}t} = -{\Big(\frac{\partial \bo{g}}{\partial {\bar{\bo{u}}}}\Big)}^{-1}\Big(  \frac{\partial \bo{g}}{\partial {\overbar{\bo{z}}}}\bar{\bo{f}} - \frac{\partial \bo{g}}{\partial {\bar{\bo{p}}}} \frac{\partial \mathcal{H}}{\partial {\overbar{\bo{z}}}}+\frac{\partial \bo{g}}{\partial t}\Big).
\end{align}
Define a new variable $\tau = t_f - t$, then $\mathrm{d}\tau/\mathrm{d}t = -1$. Set ${\bo{Z}} = {[X,Y,\Theta,\Xi]}^T$, ${\bo{P}} = {[P_X,P_Y,P_\Theta,P_\Xi]}^T$, ${\bo{U}} = {[U, \Omega, M]}^T$. Let us define a parameterized system
\begin{align}
&\left\{
  \begin{aligned}
  \label{eq:para_sys}  
     &\frac{\mathrm{d}}{\mathrm{d}\tau}{{\bo{Z}}}(\tau) = -\bar{\bo{f}}({\bo{Z}}(\tau), {\bo{U}}(\tau)),            \\
     &\frac{\rm{d}}{\rm{d}\tau}{\bo{P}}(\tau)= \frac{\partial \mathcal{H}}{\partial {\bo{Z}}}, \\ 
     &\frac{\mathrm{d}}{\rm{d}\tau}{\bo{U}}(\tau) = -{\Big(\frac{\partial \bo{g}}{\partial {\bo{U}}}\Big)}^{-1}\Big(-\frac{\partial \bo{g}}{\partial {{\bo{Z}}}}\bar{\bo{f}} +\frac{\partial \bo{g}}{\partial {{\bo{P}}}} \frac{\partial \mathcal{H}}{\partial {{\bo{Z}}}}+\frac{\partial \bo{g}}{\partial \tau}\Big).
  \end{aligned}
\right.
\end{align}
The initial value of the variables are given as
\begin{align}
    x(t_f) &= 0, y(t_f) = 0,\\
    p_\theta(t_f) &= 0, p_\xi(t_f) = 0.
\end{align}
Since the kinematics remain invariant under coordinate rotation, we can set $\theta(t_f) = 0$ in the parameterized system. The the terminal condition $\sigma(t_f) = 0$ naturally arises to guarantee that $r(t_f) = 0$, while maintaining a finite control input. Then, we can obtain the value of $\xi(t_f)$ from Eq.~(\ref{eq:satfun}) as
\begin{align}
    \xi(t_f) = -\log (1-\cos\sigma_M).
\end{align}
Let us set $\alpha$ and $\beta$ as
\begin{align}
    \alpha &= \sqrt{p_x^2(t_f) + p_y^2(t_f)},\\
    \beta &= \atantwo\Big(p_y(t_f), p_x(t_f)\Big).
\end{align}
Then, the initial conditions of Eq.~(\ref{eq:para_sys}) are
\begin{align}
  \label{eq:initcond}
&\left\{
\begin{aligned}
  \bo{Z}(0)& = {\Big[0,0,0, -\log (1-\cos\sigma_M)\Big]}^T, \\
  \bo{P}(0) & = {\Big[\alpha\cos\beta, \alpha\sin\beta,0, 0\Big]}^T,     \\
\bo{U}(0) &= \left\{ \bo{U} \in \mathbb{R} \ \big| \ \bo{g}\left( {\bo{Z}}(0), {\bo{P}}(0), \bo{U} \right) = 0 \right\}
\end{aligned}
\right.
\end{align} 

Given an FOV limit \(\sigma_M\), the solution of the parameterized system in Eq.~(\ref{eq:para_sys}) with the initial conditions specified in Eq.~(\ref{eq:initcond}) is determined by the parameters \(\alpha\) and \(\beta\). Thus, the parameterized system in Eq.~(\ref{eq:para_sys}) is fully determined by \(\alpha\), \(\beta\), and \(\tau\).
For clarity, we denote the extremal trajectories and the corresponding control generated by the parameterized system as \(\big[{{{X}}}(\alpha, \beta, \tau), {{{Y}}}(\alpha, \beta, \tau), {{{\Theta}}}(\alpha, \beta, \tau)\big]\) and $U(\alpha, \beta, \tau)$ for \(\tau \in [0, t_f]\), respectively.

\section{Real-Time Solution via NNs}\label{section:4}
This section aims at establishing a framework capable of generating
optimal commands in real time. Since the multilayer feedforward neural networks
can approximate any measurable function to any desired degree of
accuracy\cite{hornik1989multilayer}, 
we can train a feedforward neural network by the dataset of state-control pairs to approximate the mapping from the current state to the corresponding optimal guidance command. To this end, it is crucial to construct the dataset for state-control pairs, and the following subsection shall show how the parameterized system developed in Section.~\ref{section:3} allows to construct the dataset efficiently.
\subsection{Dataset Generation}
Traditionally, dataset generation for optimal control problems relies on direct and indirect numerical methods~\cite{you2021learning,Izzo2021}. However, even with the improved homotopy method, these methods still
suffer from convergence issues~\cite{WangPhysics}. To address these limitations, this work leverages the parameterized system developed in Section~\ref{section:3} to enable efficient dataset construction.

The engagement kinematics expressed in the polar coordinates \((r, \sigma)\) is 
\begin{align}\label{eq:polar}
  \begin{dcases}
    \dot{r}(t) = -\cos\sigma(t),             \\
    \dot{\sigma}(t) = \frac{\sin\sigma(t)}{r(t)} - u,
  \end{dcases}
\end{align}
where $r = \sqrt{x^2 + y^2}$ is the distance between the target and the pursuer, $\sigma$ is the lead angle, as illustrated in Fig.~\ref{fig:frame}. Since the kinematics in Eq.~(\ref{eq:kinematics}) is equivalent to that in Eq.~(\ref{eq:polar}), the optimal control command $u$ can be either determined by the triple $(x,y,\theta)$ in the Cartesian coordinates or by the pair $(r, \sigma)$ in the polar coordinates. We select the polar representation \((r, \sigma)\) for NN inputs due to its dimensional reduction advantages. Then, through the following coordinate transformation, the solution $\left[X(\alpha, \beta, \tau), Y(\alpha, \beta, \tau), \Theta(\alpha, \beta, \tau)\right]$ obtained from the parameterized system can be transformed to $\left[R(\alpha, \beta, \tau), \Sigma(\alpha, \beta, \tau)\right]$ expressed in the polar coordinates $(r, \sigma)$.
\begin{align}
    &\left\{
    \begin{aligned}
    &R(\alpha, \beta, \tau) = \sqrt{X^2(\alpha, \beta, \tau), Y^2(\alpha, \beta, \tau)},\\
    &\Sigma(\alpha, \beta, \tau) = \arccos\left(-\frac{X(\alpha, \beta, \tau)\cos\Theta(\alpha, \beta, \tau) + Y(\alpha, \beta, \tau)\sin\Theta(\alpha, \beta, \tau)}{\sqrt{X^2(\alpha, \beta, \tau) + Y^2(\alpha, \beta, \tau)}} \right)
    \end{aligned}
    \right.
\end{align}

Let us denote by $C(r,\sigma,t_g)$ the optimal guidance command at the current state $(r, \sigma)$ with a feasible time-to-go $t_g > 0$. Through the following lemma, we can reduce the data size for training.
\begin{lemma}\label{le:symmetry}
  Given a current state $(r_c, \sigma_c)$ in the polar coordinates and a feasible time-to-go $t_g>0$, the following equation holds
  \begin{align}
   C(r_c,\sigma_c,t_g) = -C(r_c,-\sigma_c,t_g).
  \end{align}
\end{lemma}
\noindent Lemma.~\ref{le:symmetry} can be derived in the same way as in \cite{wang2022nonlinear}. Thus, we omit the details of the proof here.

In view of Lemma~\ref{le:symmetry}, only the extremal trajectories with
$\sigma\in[0, \sigma_M]$ are needed. That is because the optimal guidance command of any
feasible state with $\sigma(t)\in[-\sigma_M,0]$ can be readily obtained by applying
$C(r(t),\sigma(t),t_g) = -C(r(t),-\sigma(t),t_g)$. Based on Lemma~\ref{le:align}, the propagation of the 
parameterized system should be terminated when the velocity vector of the pursuer is collinear with the target. Denote the maximum impact time by  $T(\alpha, \beta)$, we have
\begin{align}
    T(\alpha, \beta) = \min\{\tau>0 |\cos\Sigma(\alpha, \beta, \tau)=1 \}.
\end{align}
Note that $T(\alpha, \beta)$ may have a large value, leading difficulty in propagating the parameterized system in Eq.~(\ref{eq:para_sys}).
Thus, we set 
\begin{align}
    \hat{T} = \min\{\overbar{T}, T(\alpha, \beta)\},
\end{align}
where $\overbar{T}>0$ is a finite number.
Since the maximum curvature of an extremal trajectory is positively related with $\alpha$, we can set an upper bound $\bar{\alpha}$ for $\alpha$. Note that when $\tau=0$, $\mu = 0$, the system degrades to
the unconstrained optimal control problem. Given $\Theta(0) = 0$, we have
\begin{align}
  \frac{\mathrm{d}^2}{\mathrm{d}\tau^2}{\Theta}(0) = -p_x\sin\Theta(0) + p_y\cos\Theta(0) = \alpha\sin(\beta-\Theta(0)) = \alpha\sin\beta.
\end{align}
This indicates that the trajectories generated with the parameter pairs $(\alpha, \beta)$ and $(\alpha, -\beta)$ are symmetric with respect to the $x$-axis. Thus, we can set $\beta\in[0,\pi]$.
By selecting $(\alpha, \beta)$ in the field $(0,\bar{\alpha}]\times[0,\pi]$ and propagating the parameterized system in Eq.~\ref{eq:para_sys}, a large number of extremal trajectories can be generated. By leveraging the optimality and smoothness of the extremal trajectories, we set the adjustable parameter $\varepsilon=1\times 10^{-4}$. 
Then, we can establish a dataset $\mathcal{D}$ consisting of $(R(\alpha, \beta, \tau),\Sigma(\alpha, \beta, \tau), \tau)$ and $U(\alpha, \beta, \tau)$.

However, due to the finite time $\hat{T}$, the dataset does not contain all
the points of $(r, \sigma, t_g)$ and $C(r, \sigma, t_g)$. By the following lemma introduced in \cite{wang2022nonlinear},
we have that even if $\hat{T}$ is a small positive number, the NN trained by the dataset $\mathcal{D}$ is enough to represent the mapping from $(r, \sigma, t_g)$ to $C(r, \sigma, t_g)$.
\begin{lemma}\label{le:scale}
  Suppose the speed of the pursuer is $\overbar{V}$. Let $(\bar{r}({t}), \bar{\sigma}({t}))$ and $\bar{u}(t)$ for $t\in[0, \bar{t}_f]$ denotes an extremal trajectory and the corresponding control, respectively. Then, for any $t_f \in (0,\bar{t}_f)$, it follows that 
  \begin{align}
      \bar{u}(t) = \frac{\overbar{V}t_f}{\bar{t}_f}C(\frac{\bar{r}(t)t_f}{\overbar{V}\bar{t}_f}, \bar{\sigma}(t), t_f - t\frac{t_f}{\bar{t}_f})
  \end{align}
  % \begin{align}
  %   \overbar{u}({t}) = C(\hat{r} \frac{r_0}{\hat{r}_0}, \hat{\sigma} ,\hat{t}_g \frac{Vr_0}{\hat{r}_0}),
  % \end{align}
\end{lemma}
\noindent The proof of Lemma.~\ref{le:scale} is similar to that in \cite{wang2022nonlinear}, thus it is omitted here.

\subsection{Optimal Guidance via NNs}
To balance the overfitting issue and the performance of the NN, the
architecture of it is set as two hidden layers with 20 neurons in each layer. And the
activation function is chosen as the hyperbolic tangent. In the training procedure, 
the input and output are $(r(\tau), \sigma(\tau),\tau)$ and $u(\tau)$, respectively. 
Both of them are scaled to the range of $[-1,1]$ to improve the training efficiency.

Figure.~\ref{fig:nnArch} shows the usage of the NN which is denoted by
$\mathcal{N}$. In every guidance cycle, the states measured in the
Cartesian coordinates is transformed to that expressed in the polar coordinates. The input of $\mathcal{N}$ is a combination of the distance $r$, the lead angle $\sigma$ and the time-to-go $t_g$.
Then the optimal guidance command $u$ is generated in real
time by $\mathcal{N}$. 
% Considering Lemma~\ref{le:scale}, the optimal guidance command can be properly scaled to the scenarios where the pursuer has a constant speed $V$.
\begin{figure}[thpb]
  \centering
  \includegraphics[scale=0.7]{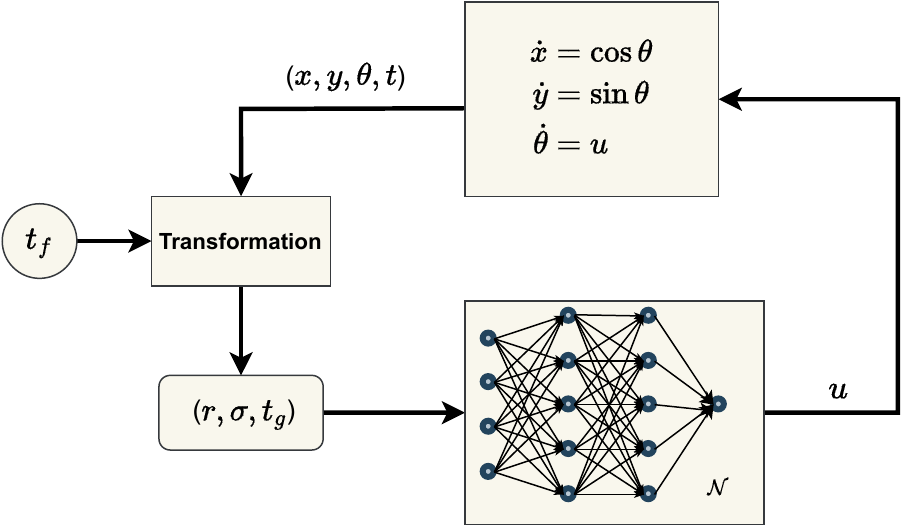}
  \caption{A schematic view of the optimal guidance.}
  \label{fig:nnArch}
\end{figure}

\section{Simulations}\label{section:5}
This section presents several numerical simulations to demonstrate the performance of the proposed guidance law. 
The guidance command generation takes approximately \SI{0.01}{\milli\second} for a given input tuple $(r, \sigma, t_g)$ on a laptop equipped with an Intel Core i5-13500H CPU. These results demonstrate that the proposed method is capable of generating optimal guidance commands in real time.
\subsection*{Case A: Numerical Examples with Different FOV Limits}
In this case, the pursuer is initialized at a position of \((0, 10000)\)~\si{m} with an initial heading angle of \(30^\circ\). The target is located at the origin. The simulation considers three FOV limits: \(30^\circ\), \(45^\circ\), and \(60^\circ\).
As shown in Fig.~\ref{fig:trajOpt}, both the GPOPS-II optimization results and the proposed guidance law generate nearly identical pursuer trajectories.
The FOV constraints are satisfied, as demonstrated in Fig.~\ref{fig:fovOpt} and Table~\ref{table:optsig}. Note that the specified and actual maximum of the lead angle are denoted by $\sigma_M$ and ${\sigma}_M'$, respectively. Table~\ref{table:optsig} also lists the control effort \(J\), the desired impact time, and the actual impact time. By incorporating the saturation function introduced in
Sec~\ref{section:3}, the guidance command profiles of the proposed guidance law are smooth, as shown in Fig.~\ref{fig:controlOpt}.
\begin{table}[ht]
  \caption{Constraints and performance index in Case A.} % title of Table
  \centering % used for centering table
  \begin{tabular}{l c c c} % centered columns (4 columns)
    \hline\hline %inserts double horizontal lines
    Parameters                          & Case 1             & Case 2             & Case 3  \\ [0.5ex] % inserts table \
    \hline % inserts single horizontal 
    $\sigma_M$, \si{deg}                & 30                 & 45                 & 60      \\ % adds vertical space
    ${\sigma}_M'$, \si{deg}          & 30.0000            & 44.9867            & 59.9979 \\
    Desired Impact Time, \si{s}         & 44                 & 50                 & 60      \\
    Actual Impact Time, \si{s}          & 44.0001            & 50.0002            & 60.0000 \\
    % $J^{\ast}$, \si{m^2/s^3}  & $1.6303\times10^3$ & $3.7612\times10^3$
                                        % & $5.9229\times10^3$                                \\
    $J$, \si{m^2/s^3} & $1.6458\times10^3$ & $3.7692\times10^3$
                                        & $5.9392\times10^3$                                \\
    % $\delta J$, \%                      & 0.3357             & 0.4320             & 0.2756  \\[1ex]
    \hline\hline  %inserts single line
  \end{tabular}\label{table:optsig} % is used to refer this table in the text
\end{table}

\begin{figure}
  \centering
  \begin{subfigure}[b]{0.45\textwidth}
    \centering
    \includegraphics[width=\textwidth]{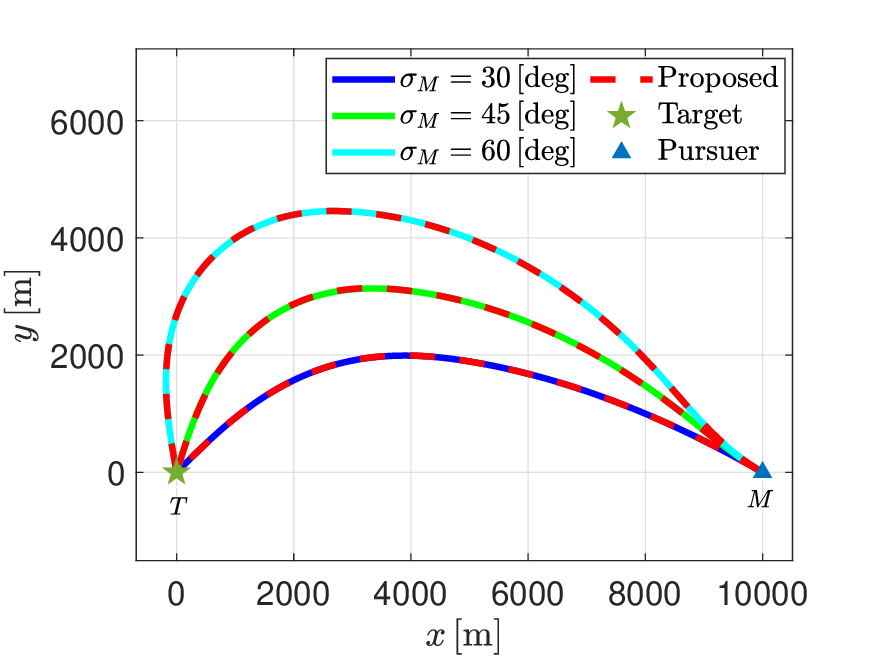}
    \caption{Pursuer Trajectories}\label{fig:trajOpt}
  \end{subfigure}
  \quad
  \begin{subfigure}[b]{0.45\textwidth}
    \centering
    \includegraphics[width=\textwidth]{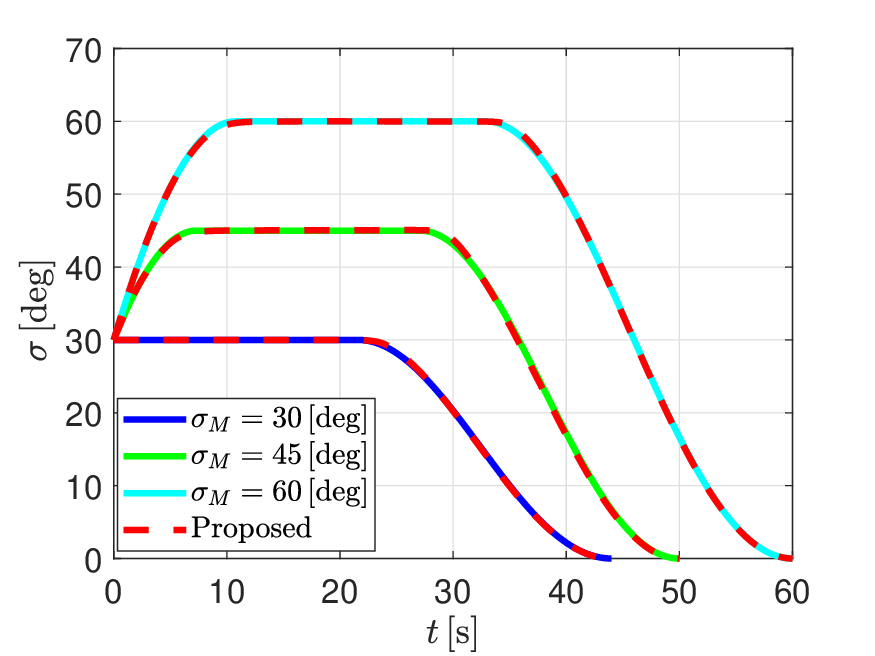}
    \caption{Lead Angle Profiles}\label{fig:fovOpt}
  \end{subfigure}
  \quad
  \begin{subfigure}[b]{0.45\textwidth}
    \centering
    \includegraphics[width=\textwidth]{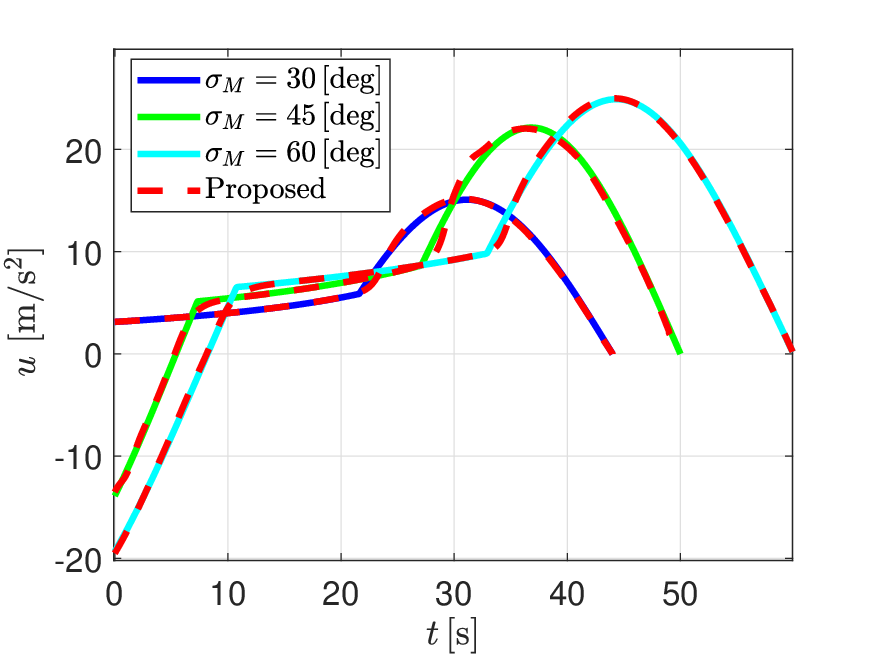}
    \caption{Control Profiles}\label{fig:controlOpt}
  \end{subfigure}
  \caption{Optimality validation with FOV limits of 30, 45 and 60 deg.}
\label{fig:optFov}
\end{figure}

\subsection*{Case B: Numerical Examples with Different Impact time}
In this subsection, the performance of the proposed guidance law is evaluated
in the scenario with a fixed FOV limit and different impact times. The results illustrated in Fig.~\ref{fig:trajTg}-\ref{fig:controlTg} show that the trajectories and guidance command profiles of the proposed guidance law nearly overlap with the optimal solutions. Fig.~\ref{fig:fovTg} demonstrates that the lead angles remain within
the FOV limit with different impact times. 
Moreover, the deviations between the actual and desired impact times are negligible, as shown in Table~\ref{table:optTg}.

\begin{table}[ht]
  \caption{Constraints and performance index in Case B.} % title of Table
  \centering % used for centering table
  \begin{tabular}{l c c c} % centered columns (4 columns)
    \hline\hline %inserts double horizontal lines
    Parameters                          & Case 1             & Case 2             & Case 3  \\ [0.5ex] % inserts table \
    \hline % inserts single horizontal 
    $\sigma_M$, \si{deg}                & 60                 & 60                 & 60      \\ % adds vertical space
    ${\sigma}_M'$, \si{deg}          & 54.3235            & 59.9742            & 59.9979 \\
    Desired Impact Time, \si{s}         & 50                 & 55                 & 60      \\
    Actual Impact Time, \si{s}          & 50.0000            & 55.0000            & 60.0000 \\
    % $J^{\ast}$, \si{m^2/s^3}  & $3.1371\times10^3$ & $4.3211\times10^3$
    %                                     & $5.9229\times10^3$                                \\
    $J$, \si{m^2/s^3} & $3.1409\times10^3$ & $4.3303\times10^3$
                                        & $5.9392\times10^3$                                \\
    % $\delta J$, \%                      & 0.1225             & 0.2128             & 0.2756  \\[1ex]
    \hline\hline  %inserts single line
  \end{tabular}\label{table:optTg} % is used to refer this table in the text
\end{table}

\begin{figure}
  \centering
  \begin{subfigure}[b]{0.45\textwidth}
    \centering
    \includegraphics[width=\textwidth]{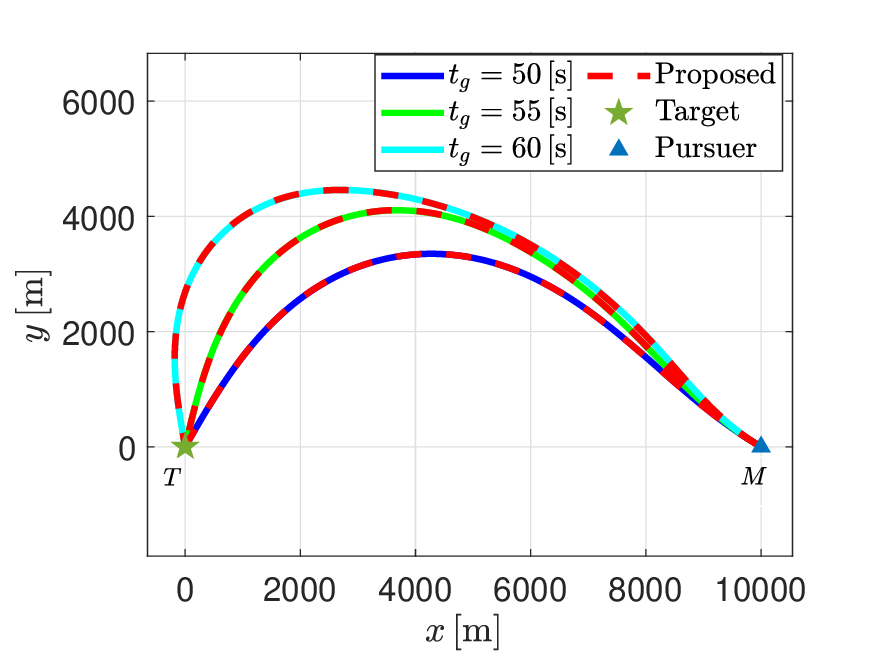}
    \caption{Pursuer Trajectories}\label{fig:trajTg}
  \end{subfigure}
  \quad
  \begin{subfigure}[b]{0.45\textwidth}
    \centering
    \includegraphics[width=\textwidth]{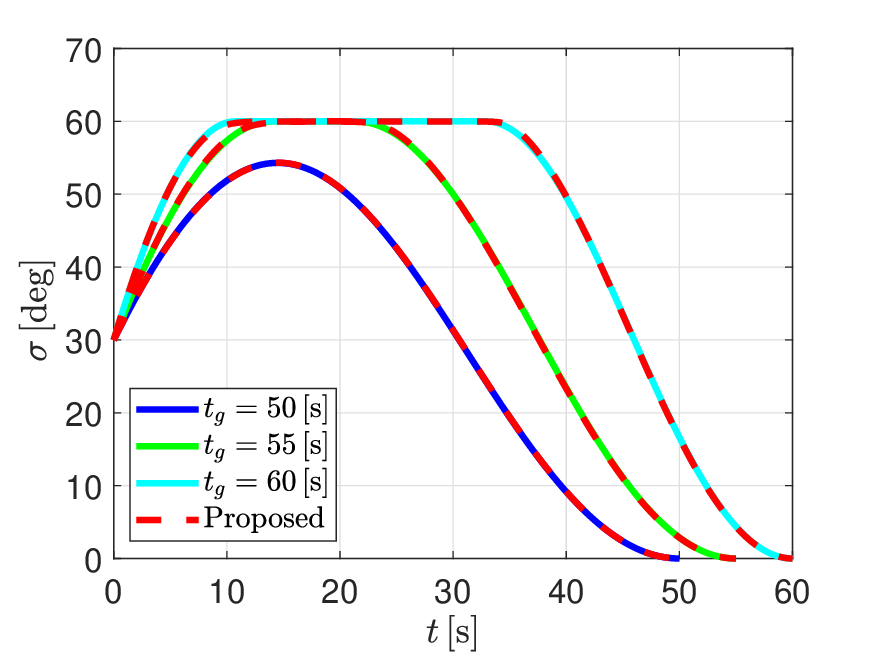}
    \caption{Lead Angle Profiles}\label{fig:fovTg}
  \end{subfigure}
  \quad
  \begin{subfigure}[b]{0.45\textwidth}
    \centering
    \includegraphics[width=\textwidth]{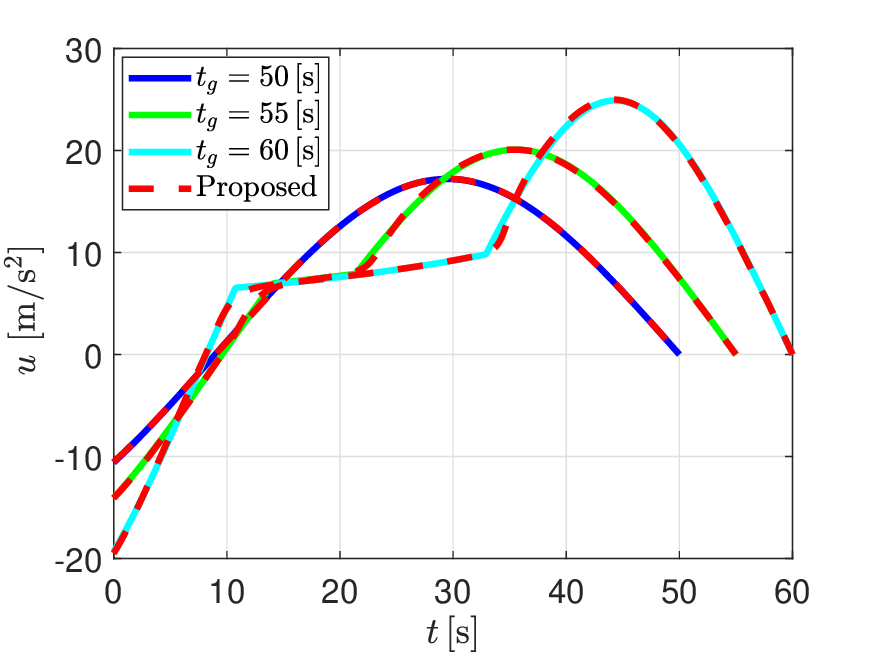}
    \caption{Control Profiles}\label{fig:controlTg}
  \end{subfigure}
  \caption{Optimality validation with FOV limits of 60 deg and tg of 50, 55, and 60 s.}\label{fig:optTg}
\end{figure}

\subsection*{Case C: Performance Comparison with Some Existing Guidance Laws}

To further analyze the proposed guidance law, its performance is compared with several existing ITCG laws with FOV limits, namely the PN-based and SMC-based guidance laws developed in~\cite{Dong2022article} and~\cite{2018Nonsingular}. The maximum guidance command is set to \SI{100}{m/s^2}.

The trajectories of the pursuers under the three different guidance laws, shown in Fig.~\ref{fig:traj60}, indicate that successful interception was achieved. Figure~\ref{fig:fov60} demonstrates that the FOV constraint was satisfied in all cases under the respective guidance laws. The guidance commands are depicted in Fig.~\ref{fig:control60}, and the corresponding control efforts are listed in Table~\ref{table:optcom}.
The PN-based method in \cite{Dong2022article} exhibits the highest control effort and generates a large initial command due to a significant time-to-go error. In contrast, the SMC-based method in \cite{2018Nonsingular} produces a smaller maximum value of the guidance command profile.
However, its guidance command profile in Fig.~\ref{fig:control60} shows that an 
abrupt command change occurred when exiting the FOV constraint due to the 
immediate change of the $\dot{\sigma}$, which may cause instability to the 
autopilot system. 

The proposed guidance law demonstrates the lowest control effort among the three methods. Moreover, its guidance command profile maintains a modest maximum value and remains smooth throughout the entire flight, thereby ensuring stable and efficient performance.
\begin{table}[ht]
  \caption{Constraints and performance index in Case C.} % title of Table
  \centering % used for centering table
  \begin{tabular}{l c c c} % centered columns (4 columns)
    \hline\hline %inserts double horizontal lines
    Parameters                  & Case 1             & Case 2             & Case 3   \\ [0.5ex] % inserts table \
    \hline % inserts single horizontal 
    Guidance Law                &PN-based~\cite{Dong2022article} & SMC-based\cite{2018Nonsingular}            & Proposed \\ % adds vertical space
    $\sigma_M$, \si{deg}        & 60                 & 60                 & 60       \\ % adds vertical space
    ${\sigma}_M'$, \si{deg}  & 60.0000            & 60.0000            & 59.9979  \\
    Desired Impact Time, \si{s} & 60                 & 60                 & 60       \\
    Actual Impact Time, \si{s}  & 59.9961            & 59.9960            & 60.0000  \\
    $J$, \si{m^2/s^3}           
    & $1.0390\times10^4$ & $6.2226\times10^3$
    & $5.9392\times10^3$                                 \\
    \hline\hline  %inserts single line
  \end{tabular}\label{table:optcom} % is used to refer this table in the text
\end{table}

\begin{figure}
  \centering
  \begin{subfigure}[b]{0.45\textwidth}
    \centering
    \includegraphics[width=\textwidth]{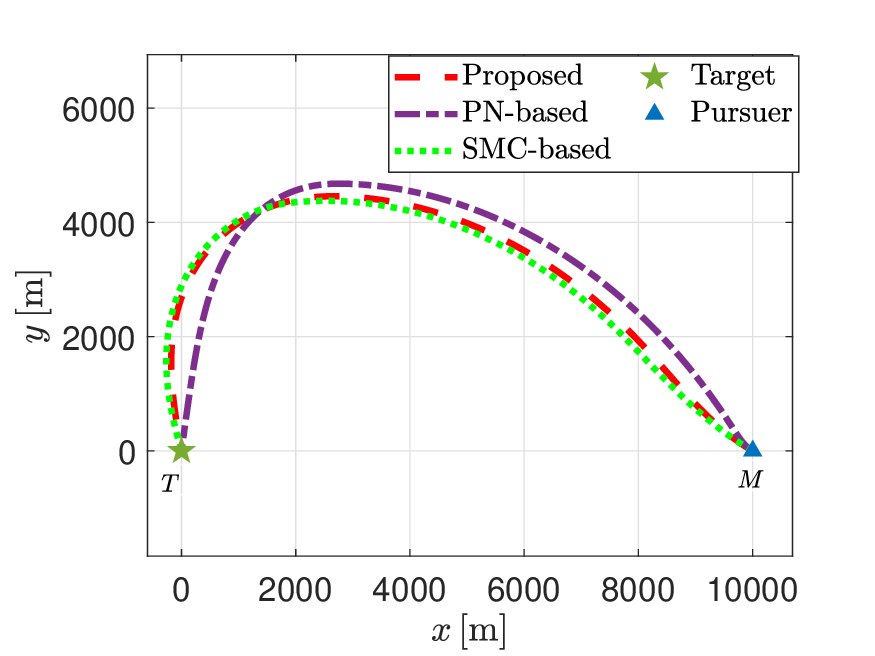}
    \caption{Pursuer Trajectories}\label{fig:traj60}
  \end{subfigure}
  \quad
  \begin{subfigure}[b]{0.45\textwidth}
    \centering
    \includegraphics[width=\textwidth]{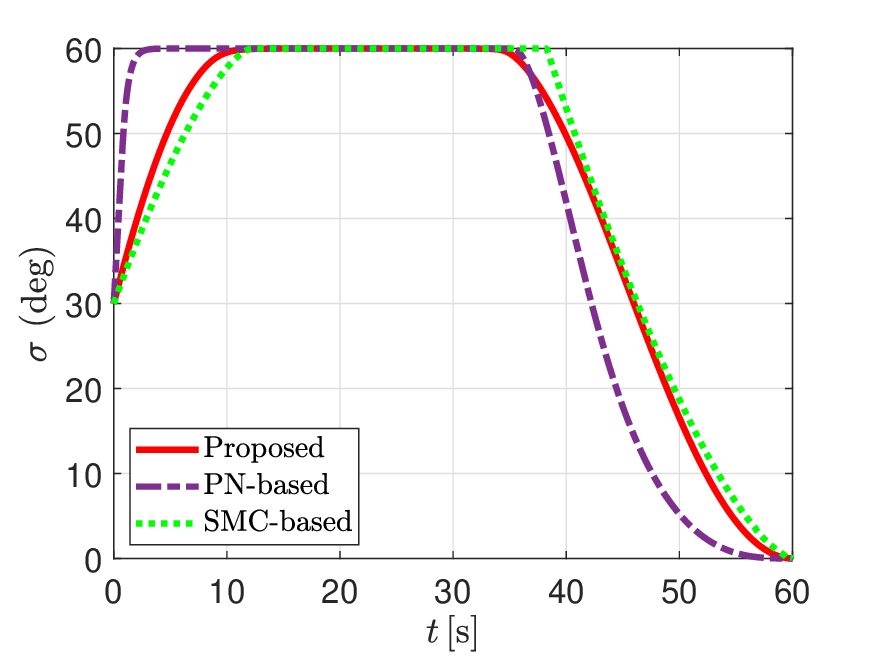}
    \caption{Lead Angle Profiles}\label{fig:fov60}
  \end{subfigure}
  \quad
  \begin{subfigure}[b]{0.45\textwidth}
    \centering
    \includegraphics[width=\textwidth]{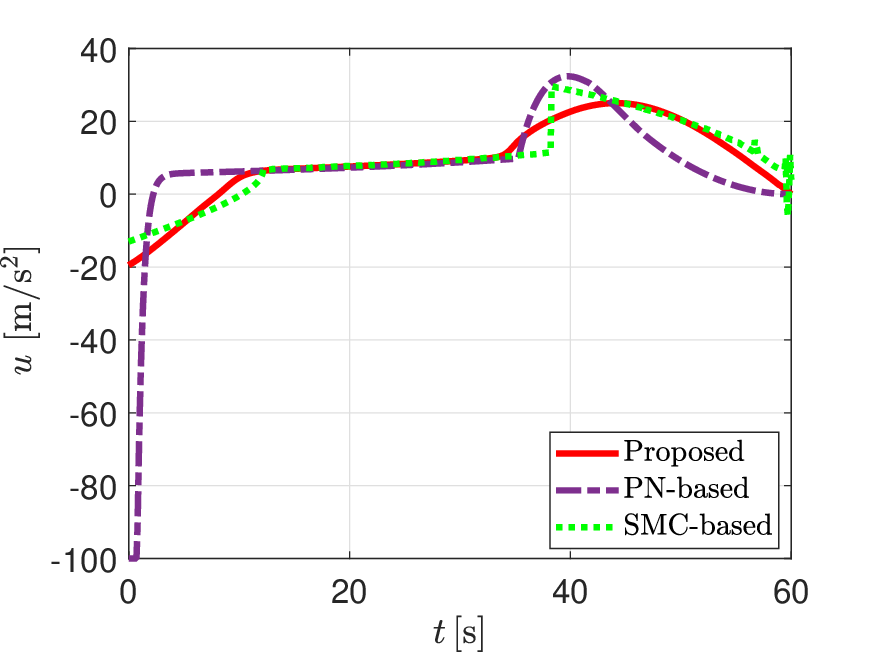}
    \caption{Control Profiles}\label{fig:control60}
  \end{subfigure}
  \caption{Performance comparison with the PN-based method~\cite{Dong2022article} and SMC-based method~\cite{2018Nonsingular}.}\label{fig:benchmark}
\end{figure}

% \subsection*{Case D: Moving Target Interception}

\section{Conclusions}\label{section:6}
In this paper, an optimal ITCG law with the FOV limit is proposed. First, the original inequality-constrained nonlinear OCP is converted 
into an equality-constrained OCP with a saturation function. 
By applying the PMP, the necessary conditions are derived and a parameterized system is established, ensuring that its solution space matches that of the nonlinear OCP with the FOV constraint.
A dataset containing the state, time-to-go and optimal guidance command is generated by simply propagating this parameterized system.
This dataset enables the training of an NN which maps the state and time-to-go to the optimal guidance command.
The trained NN can generate an optimal command within 0.1 ms without violating the FOV constraint.
Several numerical simulations have shown that the proposed guidance law is nearly optimal, which indicates superior performance than some existing methods. 
Future research may extend this methodology to three-dimensional engagement scenarios, further exploring its potential applications and improvements.

\section*{Declaration of Competing Interest}
The authors declare that they have no known competing financial interests or personal relationships that could have appeared to influence the work reported in this paper.

\section*{Acknowledgements}
This research was supported by the National Natural Science Foundation of China
under grant No. 62088101.
% \begin{appendix}
% \renewcommand{\thefigure}{\Alph{section}\arabic{figure}} 
% \renewcommand{\theequation}{A\arabic{equation}}
% \setcounter{equation}{0}  % reset counter 
\appendix
\section{Proof of Lemmas}\label{appendix:A}

\begin{proof}[Proof of Lemma~\ref{le:align}:]
By contradiction, assume that we have an extremal trajectory start from a point P to the origin and is denoted by $\Gamma_{PO} = {x(t), y(t), \theta(t)}$ with $t\in[0, t_f]$. There exists a point $A$ at time $t_{col} \in (0, t_f)$ such that the velocity vector is collinear with the LOS. Let $\Gamma_{PA}$ be the sub-trajectory of $\Gamma_{PO}$ from P to A, and $\Gamma_{AO}$ be the sub-trajectory of $\Gamma_{PO}$ from A to the origin. Denote by $u_A$ the corresponding extremal control of $\Gamma_{PO}$. Set $u_A^{\prime} = -u_A$, then a symmetric trajectory $\Gamma^{\prime}_{AO}$ is obtained. Let $\Gamma_{PA}$ be the sub-trajectory of $\Gamma_{PO}$ from P to A, and $\Gamma^{\prime}_{PO}$ be the concatenation of $\Gamma_{PA}$ and $\Gamma^{\prime}_{AO}$. Then, it is apparent that 
 the cost of $\Gamma^{\prime}_{PO}$ equals to that of $\Gamma_{PO}$. In general cases, $u_{col}\neq 0$, then the control profile of $\Gamma^{\prime}_{PO}$ is discontinuous at point A. This violates the necessary condition where the control is continuous. Then, $\Gamma^{\prime}_{PO}$ is not an extremal trajectory, so as $\Gamma_{PO}$. This completes the proof.
\end{proof}

\begin{proof}[Proof of Lemma~\ref{lemma:jacobian}:]
  To prove the non-singularity of the Jacobian $\partial \bo{g}/\partial \bar{\bo{u}}$, we proceed as follows:
  \begin{align}
       \frac{\partial \bo{g}}{\partial \bar{\bo{u}}} 
       &=  \begin{bmatrix}
       -1 & 0  & -\sin\sigma  \\
       0 & -\varepsilon & -\exp(-\xi) \\
        -\sin\sigma & -\exp(-\xi) & 0
       \end{bmatrix}, 
  \end{align}
  The determinant of $\partial \bo{g}/\partial \bar{\bo{u}}$ is
  \begin{align}
       \det\Big(\frac{\partial \bo{g}}{\partial \bar{\bo{u}}}\Big) &= \exp(-2\xi) + \varepsilon \sin^2\sigma > 0.
  \end{align}
  Therefore, we have that $\partial \bo{g}/\partial \bar{\bo{u}}$ is non-singular, which completes the proof. 
  \end{proof}
% \end{appendix}



%%       \citet{<label>} ==> Jones et al. [21]
%%       \citep{<label>} ==> [21]
%%

\bibliographystyle{elsarticle-num-names} 


\bibliography{bib}
% \begin{thebibliography}{}



% \end{thebibliography}
\end{document}